\documentclass[12pt]{article}
\usepackage{amssymb}
\usepackage{amsmath}
\usepackage{amsthm}
\usepackage{mathtools}
\usepackage{enumitem}
\usepackage{xcolor}

\newcommand{\Ac}{\mathcal{A}}
\newcommand{\Jc}{\mathcal{J}}

\numberwithin{equation}{section}

\newtheorem{theorem}{Theorem}[section]
\newtheorem{proposition}[theorem]{Proposition}
\newtheorem{lemma}[theorem]{Lemma}
\newtheorem{corollary}[theorem]{Corollary}
\newtheorem{remark}[theorem]{Remark}

\newcommand{\Fc}{\mathcal{F}}
\renewcommand{\S}{\mathcal{S}}
\newcommand{\Ru}{\mathcal{R}}
\newcommand{\Uc}{\mathcal{U}}
\newcommand{\Lc}{\mathcal{L}}
\newcommand{\C}{\mathbb{C}}
\newcommand{\Z}{\mathbb{Z}}
\newcommand{\End}{\mathrm{End}}
\newcommand{\gm}{\gamma}
\newcommand{\dt}{\delta}
\newcommand{\ot}{\otimes}
\newcommand{\eps}{\epsilon}
\newcommand{\al}{\alpha}
\newcommand{\bt}{\beta}

\begin{document}

\title{Graded twisted reflection matrices of $GL$-type}
\author{D. Algethami${}^{\dag}$, A. Mudrov${}^{\ddag}$
\vspace{10pt}\\
\small ${\dag}$ Department of Mathematics, College of Science,\\
\small University of Bisha, P.O. Box 551, Bisha 61922, Saudi Arabia,
\vspace{10pt}\\
\small ${\ddag}$ Moscow Institute of Physics and Technology,\\
\small 9 Institutskiy per., Dolgoprudny, Moscow Region, 141701, Russia,
\vspace{10pt}\\
\small e-mail: dalgethami@ub.edu.sa, mudrov.ai@mipt.ru
}
\date{}
\maketitle

\begin{abstract}
We classify constant numerical solutions of the  $\Z_2$-graded  twisted
reflection equation for the fundamental $GL$-type R-matrix and  an arbitrary
parity sequence, without assuming evenness or invertibility.  The classification is based on a four-index locality principle
established for this type of reflection equation.  We determine the singular and invertible
solutions and recover previously conjectured matrices for the  minimal
symmetric grading.  Even invertible solutions comply with the
recent classification of graded Satake diagrams.
The classical stabilizer of a full-support lower-triangular non-even family is a
Lie superalgebra whose even part is an electrical Lie algebra.
\end{abstract}

\par
\medskip
{\small\noindent\hangindent=5.4em\hangafter=1
\par}

\section{Introduction}
The term reflection equation (RE) refers to several quadratic identities
involving a matrix $K$ and a fixed solution $R$ of the Yang--Baxter
equation.  These identities first appeared as consistency conditions for
factorized scattering on a half-line and became a standard algebraic tool for
integrable systems with open boundary conditions \cite{Cherednik,Sklyanin}.
Their constant forms define covariant quadratic algebras
\cite{KulishSklyanin,KulishSasakiSchwiebert} and occur in applications including
equivariant deformation quantization \cite{DoninMudrovEQ}.  Even for the standard R-matrices coming
from quantum groups, the classification of numerical solutions is a
long-standing problem.  It was solved for the non-twisted reflection equation
of the general linear quantum group in \cite{M1} and for its {$\Z_2$-graded} counterpart
in \cite{AMS24}.  The present paper addresses numerical solutions of the
twisted reflection equation for general linear quantum supergroups.

Solutions of the reflection equation, often called K-matrices, are also related
to coideal subalgebras and quantum symmetric pairs \cite{NS,NDS}.  Their
systematic construction goes back to Letzter and subsequent developments
\cite{Letzter,KolbRE,Kolb}.  Like universal R-matrices, universal K-matrices
produce solutions of the reflection equation in individual representations.
Such elements have been constructed for a broad range of non-graded quantum
groups in \cite{BalagovicKolb,RegelskisVlaar,AppelVlaar}.  The {universal} and
representation-theoretic approaches thus describe the same object at two
levels: the {latter} gives explicit matrices in a fixed representation, whereas
the {former} seeks their representation-independent origin.

The theory of quantum supergroups is a well-established area of quantum
algebra; see, for instance, \cite{Yamane}.  Graded reflection equations and
their applications to integrable models have also been studied for quite some time
\cite{Arnaudon,LimaSantos}.  A systematic representation-theoretic  approach to
 the supersymmetric reflection equation has developed more recently through
quantum supersymmetric pairs \cite{Shen,ShenWang}, quantum super-spherical
pairs \cite{AMS}, and graded Satake diagrams \cite{AMSsatake}.  The diversity
of non-isomorphic Borel subalgebras gives rise to a corresponding diversity of
supersymmetric pairs and potential K-matrices.  Consequently, different
parity sequences can lead to genuinely different problems and cannot in general be reduced to a mere relabelling.

Constant solutions of the non-twisted $\Z_2$-graded reflection equation of
$GL$-type were classified for every  parity sequence  in \cite{AMS24}.  They are
all even and can be obtained from the non-graded list by the evenness restriction imposed by the
parity sequence.  The twisted equation behaves differently: it allows for non-even solutions.  For a constant parity sequence, the non-graded twisted reflection equation is treated in \cite{NS}.  A general classification without
evenness or invertibility restrictions remained open.

We provide {such a} classification for the fundamental $GL$-type R-matrix \cite{Zhang} and an
arbitrary  parity sequence. The deformation parameter $q$ can take any nonzero complex value other than a 4th root of unity.  The first difference from the non-twisted case is
already visible in dimension two: when the two basis vectors have opposite parities, a nonzero odd entry may
occur below the diagonal and is quadratically related to the diagonal entries.
In higher dimensions this phenomenon produces two non-even families in
addition to the even ones.

The {classification} rests on  two complementary locality facts.  First, the equation
restricts to every principal subspace: the corresponding principal block of a
K-matrix satisfies the same reflection equation with the induced
grading (Lemma~\ref{lem:locality}).  Conversely,  the full equation is  determined by its principal
subspaces  of dimension at most four, since every residual coefficient involves
at most four basis indices {(Corollary~\ref{cor:fourlocal})}.  We refer to the latter statement as the four-index
locality principle.  The analysis is cumulative.   Dimension  two identifies the
 possible  local blocks,  dimension  three determines how a third index can be
attached, and  dimension  four rules out incompatible configurations on disjoint
supports.  These local restrictions assemble into a global structure theorem {(Theorem~\ref{thm:structureall})},
from which the singular and invertible classifications follow.   Every invertible
even solution is accounted for, at the classical level, by a type~A graded
Satake diagram from \cite{AMSsatake}; see Subsection~\ref{sec:satakelower}.
For the minimal symmetric  grading $(1^m,0^N,1^m)$, the obtained result
recovers the two even families conjectured in \cite{AMS}.

The global classification separates {K-matrices} having a nonzero entry above the diagonal
from lower-triangular solutions.  The former are necessarily even, whereas the lower-triangular
class contains both even and non-even solutions.  The precise support conditions and
parametrizations are stated in Section~\ref{sec:structure}.

{
There are two lower-triangular full-support (all lower-triangular entries nonzero) families of solutions.
One of them corresponds to a non-graded Satake diagram of all white nodes with no arcs, see Subsection \ref{sec:satakelower}.
 This family comprises  even matrices that are fixed at $q=1$
by the so-called electrical Lie subalgebras of type $A$, see   \cite{BGG}.
 These algebras originate 
in electrical-network theory \cite{LamPylyavskyy}; related electrical varieties
admit an interpretation as vertex integrable statistical models
\cite{GorbounovTalalaev}.

K-matrices from the other lower-triangular  family are not even.  All but one of the indices have the same parity; let $n$ be their number. In the classical limit, such a matrix is fixed by  a subalgebra $\mathfrak{k}$ in $\mathfrak{g}\simeq \mathfrak{gl}(1|n)$ whose
even part is isomorphic to an electrical Lie subalgebra.
The Lie superalgebra $\mathfrak{k}$ is isomorphic to $\mathfrak{osp}(1|n-1)$ when $n$ is odd. For even $n$, the corresponding skew-symmetric form is degenerate with a one-dimensional kernel, while $\mathfrak{k}$ has a one-dimensional even center.  This gives a natural
supersymmetric extension of electrical Lie algebras of type $A$, see Section~\ref{sec:superelectrical}.

}

\section{ Notation and the  graded twisted reflection equation}\label{sec:conv}
{Throughout the paper, $q\in\C^\times$ and $q^4\ne1$.}
Let $V=\C^d$, where $d\ge1$, have an ordered homogeneous basis
$v_1,\ldots,v_d$ with arbitrary parity
$$
 p_i=|i|\in\{0,1\},\qquad s_i=(-1)^{p_i}.
$$
{Write $I=\{1,\ldots,d\}$ for the set of basis indices.}
The order of the basis and the parity sequence $p_1,\ldots,p_d$ are fixed
throughout.  Whenever a subset of basis indices is considered, it is kept in
the order inherited from ${I}$, with the corresponding restricted
parities.   For $\eps\in\{0,1\}$ set
$$
 I_\eps=\{i\mathrel{\in}{I}:p_i=\eps\}.
$$
The grading will be called mixed if $I_0\ne\varnothing$ and $I_1\ne\varnothing$.
 For $i<j$ we write
$$
{I_{ij}}=\{r\mathrel{\in}{I}:i<r<j\}
$$
for the indices lying strictly between $i$ and $j$.
For $S\subseteq {I}$, the submatrix
$K_S=(k_{ij})_{i,j\in S}$, with the indices kept in their inherited order,
will be called the \emph{principal block} of $K$ on $S$.

 We use the standard matrix units
$$
 e_{ij}v_k=\dt_{jk}v_i,\qquad A=\sum_{i,j}A_{ij}e_{ij}.
$$
The matrix $e_{ij}$ is homogeneous of degree
$p_i+p_j$ modulo $2$.  Thus the component $A_{ij}e_{ij}$ is even when
$p_i=p_j$ and odd when $p_i\ne p_j$.   A condition such as
$p_i=p_j=\eps$ refers to the parities of the indices; it is not a statement
about  the degree of the  matrix unit $e_{ij}$.
All tensor products of endomorphisms are taken in the graded algebra.  For homogeneous endomorphisms
$x_1,x_2,y_1,y_2$, where $x_2$ and $y_1$ have degrees $|x_2|$ and $|y_1|$, respectively, one has
\begin{equation}\label{eq:gradedprod}
 (x_1\ot x_2)(y_1\ot y_2)
 =(-1)^{|x_2||y_1|}(x_1y_1)\ot(x_2y_2).
\end{equation}
This multiplication is extended bilinearly to arbitrary endomorphisms.  We use
{the} standard leg notation $K_1=K\ot1$ and $K_2=1\ot K$.

Let $P:V\ot V\to V\ot V$ be the graded flip,
$$
 P(v_i\ot v_j)=(-1)^{p_ip_j}v_j\ot v_i,
$$
and write $X_{21}=PXP$ and $X_{12}=X$ for
$X\in\End(V\ot V)$.

{Put}
\begin{equation}\label{eq:q}
 \omega=q-q^{-1}\in\C^\times.
\end{equation}
 For $\eps\in\{0,1\}$ we shall use
\begin{equation}\label{eq:gamma}
{\sigma_\eps}=(-1)^\eps,\qquad
 \gm_\eps=1+q^{-{\sigma_\eps}}.
\end{equation}
Thus $s_i={\sigma_{p_i}}$.  We also fix $\eta\in\C$ such that
\begin{equation}\label{eq:eta}
 \eta^2=-\omega=-(q-q^{-1}).
\end{equation}
Replacing $\eta$ by $-\eta$ amounts only to a change of sign of corresponding free parameters
in non-even solutions.

 The fundamental graded {R-matrix of $GL$-type} (in the opposite triangular convention to that of \cite{Zhang}) is
\begin{equation}\label{eq:R}
 R=\sum_{i,j=1}^d q^{s_i\dt_{ij}}e_{ii}\ot e_{jj}
 +\omega\sum_{j<i}(-1)^{p_j}e_{ij}\ot e_{ji}.
\end{equation}
The super-transposition is {defined as}
\begin{equation}\label{eq:st}
 (A^t)_{ij}=A_{ji}(-1)^{(p_i+p_j)p_i}.
\end{equation}
{For homogeneous $A,B\in\End(V)$, it satisfies
$$
 (AB)^t=(-1)^{|A||B|}B^tA^t,
$$
so $t$ is a graded anti-automorphism of $\End(V)$.}
The partial super-transpositions are applied to tensor factors, without an additional sign:
$$
 (x\ot y)^{t_1}=x^t\ot y,
 \qquad
 (x\ot y)^{t_2}=x\ot y^t.
$$
With these conventions $R^{t_1t_2}=R_{21}${. The} twisted reflection equation {for a matrix $K=(k_{ij})\in\End(V)$} is
\begin{equation}\label{eq:TRE}
 R_{21}K_1R_{21}^{t_1}K_2=K_2R_{12}^{t_2}K_1R_{21}.
\end{equation}
We call $K$ \emph{even} {in the usual graded sense} when $k_{ij}=0$ for $p_i\ne p_j$.

For a fixed matrix $K$, write
$$
 \S=\S(K)=\{(i,j):k_{ij}\ne0\}
$$
for its \emph{support}, and set
$$
 \S_0=\{(i,j)\in\S:i=j\},\qquad
 \S_+=\{(i,j)\in\S:i<j\},\qquad
 \S_-=\{(i,j)\in\S:i>j\}.
$$
These are the diagonal, strictly upper, and strictly lower supports, and
$\S=\S_0\sqcup\S_+\sqcup\S_-$.  We also write
$$
 \Uc=\S_0\sqcup\S_+,
 \qquad
 \Lc=\S_0\sqcup\S_-
$$
for the upper and lower supports; this notation will be used throughout.  In
particular, $K$ is upper triangular if
$\S=\Uc$, lower triangular if $\S=\Lc$, and
strictly lower triangular if $\S=\S_-$.   For an off-diagonal entry $k_{ij}$,
we call $i$ and $j$ its endpoints.

\begin{remark}\label{rem:nohiddenq}
Scalar factors divided out in arguments below are products of powers of
$q$, $q\pm1$, and $q^2+1$.  We cancel any such nonzero common factor
and may then normalize by a factor $\pm q^m$.
\end{remark}

\section{Locality principle }\label{sec:local}
Define the quadratic residual
$$
 \Fc(K)=R_{21}K_1R_{21}^{t_1}K_2-K_2R_{12}^{t_2}K_1R_{21}.
$$
{We denote by} $\Fc_{ab}^{ij}(K)$ the coefficient of
$v_a\ot v_b$ in $\Fc(K)(v_i\ot v_j)$.
{The R-matrices act on $V\ot V$ by the following formulas:}
\begin{equation}\label{eq:R21action}
 R_{21}(v_i\ot v_j)=
 \begin{cases}
 q^{s_i}v_i\ot v_i,&i=j,\\
 v_i\ot v_j+\omega(-1)^{p_ip_j}v_j\ot v_i,&i>j,\\
 v_i\ot v_j,&i<j,
 \end{cases}
\end{equation}
and, for $i\ne j$,
\begin{equation}\label{eq:toff}
 R_{21}^{t_1}(v_i\ot v_j)=R_{12}^{t_2}(v_i\ot v_j)=v_i\ot v_j,
\end{equation}
while
\begin{align}
 R_{21}^{t_1}(v_i\ot v_i)
 &=q^{s_i}v_i\ot v_i+\omega s_i\sum_{a>i}v_a\ot v_a,\label{eq:t1}\\
 R_{12}^{t_2}(v_i\ot v_i)
 &=q^{s_i}v_i\ot v_i+\omega\sum_{a>i}s_a v_a\ot v_a.\label{eq:t2}
\end{align}
{These formulas facilitate reduction of the RE as explained below.}

For $S\subseteq {I}$, set
$V_S=\operatorname{span}\{v_i:i\in S\}$ with the {induced} grading, and let
$P_S:V\to V_S$ be the coordinate projection.  {The principal block on $S$ is}
$$
 K_S=P_SK|_{V_S}=(k_{ij})_{i,j\in S}.
$$

\begin{lemma}\label{lem:locality}
{One has}
\begin{equation}\label{eq:locality}
 (P_S\ot P_S)\,\Fc_V(K)|_{V_S\ot V_S}=\Fc_{V_S}(K_S),
\end{equation}
where the right-hand side is defined by the same formulas
\textup{(\ref{eq:R}--\ref{eq:TRE})} using the induced grading on $V_S$.
\end{lemma}

\begin{proof}
The operator $R_{21}$ only preserves or exchanges the two indices already present.  New
indices can enter through partially transposed R-matrices applied to a diagonal tensor.  If such an
index lies outside $S$, one copy survives all subsequent factors and is removed by the final
projection.  The same happens when one of the $K$-factors sends an index {from} $S$ outside $S$:
such outside indices appear only in non-diagonal terms.  The remaining factors cannot bring both tensor components back to $S$
{(because the other copy of $K$ acts on the different tensor factor)}.  After projection,
therefore, only matrix entries with both indices in $S$ remain, and the surviving transpose
sums are exactly those of the induced grading.  This is the residual for $K_S$.
\end{proof}

Let $S\subseteq I$ and let $A\in\End(V_S)$ solve the twisted
reflection equation for the induced  grading on $V_S$.  Extend $A$ by zero to an operator
$\widetilde A\in\End(V)$, i.e. $\widetilde A v_i=0$ for $i\notin S$ and
$\widetilde A|_{V_S}=A$.

\begin{lemma}\label{lem:zeroext}
The matrix $\widetilde A$ solves \eqref{eq:TRE} on $V$.
\end{lemma}

\begin{proof}
If one input index lies outside $S$, one of the two factors
$\widetilde A_1,\widetilde A_2$ annihilates the corresponding term.  A {diagonal tensor} created outside $S$ by a {partially transposed R-matrix}
is annihilated for the same reason.  When both input indices are in $S$, the surviving part of the calculation
is precisely the twisted equation for $A$.  Hence the zero extension is again a solution.
\end{proof}

\begin{corollary}\label{cor:fourlocal}
A matrix $K$ solves \eqref{eq:TRE} if and only if every principal submatrix
$K_S$ with $1\le |S|\le4$ solves the twisted reflection equation for the induced grading on $V_S$.
\end{corollary}

\begin{proof}
Necessity follows from Lemma~\ref{lem:locality}.  For the converse, fix a coefficient of $\Fc(K)$ with input indices $c,d$ and output indices $a,b$.  It is already a coefficient of the residual on the principal set $S=\{a,b,c,d\}$, whose cardinality is at most four.  That residual vanishes by assumption.
\end{proof}

Lemma~\ref{lem:locality} is the principal restriction property, while Corollary~\ref{cor:fourlocal} is the four-index locality principle used below.

Let us give an algebraic interpretation of
Lemmas~\ref{lem:locality} and~\ref{lem:zeroext}.  For each $S\subseteq I$ denote by $\Ac_S$ the
quotient of the free associative algebra generated by the entries of
$K_S=(K_{ij})_{i,j\in S}$ by the ideal generated by the coefficients of
$\Fc_S(K_S)$.  For $S=I$ it is the reflection equation algebra associated
with \eqref{eq:TRE}.

\begin{proposition}
For each $S\subseteq I$, the algebra $\Ac_S$ is embedded in $\Ac_I$ as a
retract.
\end{proposition}

\begin{proof}
First of all observe that Lemmas~\ref{lem:locality} and~\ref{lem:zeroext}
generalize to the matrix $K=K_I$ whose entries generate $\Ac_I$.
Denote by $\Jc_S\subset \Ac_I$ the two-sided ideal generated by those
$K_{ij}$ for which either $i\notin S$ or $j\notin S$.  Then
Lemma~\ref{lem:zeroext} implies a short exact sequence
\[
0\longrightarrow \Jc_S\longrightarrow \Ac_I\longrightarrow \Ac_S
\longrightarrow 0,
\]
which splits by Lemma~\ref{lem:locality} via the natural embedding
$\Ac_S\to \Ac_I$ induced by $S\hookrightarrow I$.
\end{proof}

Numerical K-matrices are one-dimensional representations of the
reflection equation (non-graded) algebras.
In particular, Lemmas~\ref{lem:locality} and~\ref{lem:zeroext} follow
from the embedding and quotient maps between $\Ac_S$ and $\Ac_I$.

\section{ Reflection equation in dimension  two}\label{sec:rank2}
The reflection equation in dimension one is satisfied with an
arbitrary scalar, so the first nontrivial dimension is two.
{For} $i<j$, write the principal block {on $\{i,j\}$} as
$$
 \begin{pmatrix}a&b\\c&d\end{pmatrix}.
$$

\begin{lemma}\label{lem:twoall}
\begin{enumerate}[label=\textup{(\roman*)}]
\item If $p_i=p_j$, the complete two-index equation is
equivalent to
\begin{equation}\label{eq:same2}
 ab=0,\qquad bd=0,\qquad b(c+q^{-s _i }b)=0.
\end{equation}
\item If $p_i\ne p_j$, the complete two-index equation is equivalent to
\begin{equation}\label{eq:mixed2local}
 b=0,\qquad c^2+\omega ad=0.
\end{equation}
\end{enumerate}
\noindent Consequently every nonzero odd entry lies strictly below the
diagonal and forces both endpoint diagonal entries to be nonzero.
\end{lemma}

\begin{proof}
By the locality principle, the calculation may be done in dimension two.  Using \textup{(\ref{eq:R21action}--\ref{eq:t2})}, the case $p_i=p_j$ gives
\eqref{eq:same2}.  When $p_i\ne p_j$, there are
$$
 \pm(q+q^{-1})b^2,
 \qquad
 \pm(q+q^{-1})(c^2+\omega ad)
$$
among the residual coefficients.
Since $q^2\ne-1$, the first coefficient forces $b=0$, and the second gives
$c^2+\omega ad=0$.  Conversely, after imposing these two relations, direct
substitution makes every residual coefficient vanish.
\end{proof}

\begin{lemma}\label{lem:two}
If $K$ is {an even solution of \eqref{eq:TRE}}, then for $p_i\ne p_j$ one has
\begin{equation}\label{eq:oppdiag}
 k_{ii}k_{jj}=0,
\end{equation}
and if $i<j$, $p_i=p_j$ and $k_{ij}\ne0$, then
\begin{equation}\label{eq:qpair}
 k_{ii}=k_{jj}=0,\qquad k_{ji}=-q^{-s_i}k_{ij}.
\end{equation}
\end{lemma}

\begin{proof}
For an even matrix a mixed two-index block is diagonal, and \eqref{eq:mixed2local} gives \eqref{eq:oppdiag}.  The statement for $p_i=p_j$ follows directly from \eqref{eq:same2}.
\end{proof}

When $i<j$, $p_i=p_j$, and $k_{ij}\ne0$, relation \eqref{eq:qpair} fixes the accompanying lower entry $k_{ji}$.  We call these two off-diagonal entries a \emph{skew pair}, or simply the pair, on the endpoints $i,j$.

A nonzero lower entry $k_{ji}$, $i<j$, for which $k_{ij}=0$ will be called
\emph{lower-only}.  In the local statements below, prescribed entries are called
compatible when they can occur together in a solution of the induced principal-block
equation; otherwise they are incompatible.

\section{ Reflection equation in dimension  three}\label{sec:rank3}
{With every K-matrix we associate a partially oriented graph on the vertex set $I$.
A nonzero off-diagonal entry $k_{ij}$ gives an arrow $i\leftarrow j$, and a nonzero
diagonal entry $k_{ii}$ gives an unoriented loop at $i$.
For a lower-triangular matrix, every arrow goes from a smaller index to a larger one.
When only the support matters, we omit the arrowheads, write $i-\!\!\!-j$ for the edge
corresponding to $(j,i)\in\S_-$ (we use $i-\!\!\!-j$ only with $i<j$, i.e. for lower-triangular matrices), and represent a loop at $i$ by a mark.}
{We call a vertex $i$ \emph{marked} if $k_{ii}\ne0$.}
Call a vertex
\emph{active} if it is marked or incident with an edge.  A set $S$ of
vertices is called a \emph{marked clique} precisely when
$$
 k_{ii}\ne0\quad(i\in S),\qquad
 k_{ji}\ne0\quad(i<j,\ i,j\in S).
$$

\begin{lemma}\label{lem:allclosure}
Let $K$ be a lower-triangular solution, not necessarily even.
On every ordered triple $i<j<k$, the residual coefficient equations contain relations
of the form
\begin{equation}\label{eq:cl}
 \al_1\,k_{ii}k_{kj}=\bt_1\,k_{ji}k_{ki},\qquad
 \al_2\,k_{ji}k_{kj}=\bt_2\,k_{jj}k_{ki},\qquad
 \al_3\,k_{ji}k_{kk}=\bt_3\,k_{ki}k_{kj}.
\end{equation}
where all $\al_r,\bt_r$ are nonzero scalars.  Hence a nonzero diagonal at one vertex together with
a nonzero edge joining the other two vertices forces $\{i,j,k\}$ to be a
marked clique.
\end{lemma}

\begin{proof}
Expanding the reflection equation for a lower-triangular matrix
on three indices gives the displayed
relations with the following coefficients:
$$
\begin{array}{c|ccc}
(p_ip_jp_k)&(\al_1,\bt_1)&(\al_2,\bt_2)&(\al_3,\bt_3)\\ \hline
000&(1+q,q)&(q,1+q)&(1+q,q)\\
001&(1+q,q)&(q,1+q)&(q-1,-1)\\
010&(1+q,q)&(1,1-q)&(1+q,q)\\
011&(q-1,-q)&(1,1+q)&(1+q,1)\\
100&(q-1,-1)&(q,1+q)&(1+q,q)\\
101&(1+q,-1)&(q,q-1)&(1+q,-1)\\
110&(1+q,1)&(1,1+q)&(q-1,-q)\\
111&(1+q,1)&(1,1+q)&(1+q,1)
\end{array}
$$
The assumption $q^2\ne1$ makes every entry in the table nonzero.  The stated
conclusion now follows directly from
\eqref{eq:cl}.
\end{proof}

{For} $i<j<k$ {with} $p_i=p_j=p_k=\eps$, put $d_r=k_{rr}$ and, for $b>a$,
$$
 h_{ba}=\gm_\eps^{-1}k_{ba}.
$$

\begin{proposition}\label{prop:rank3hom}
Let $i<j<k$ with $p_i=p_j=p_k=\eps$.
For a lower-triangular principal block on $\{i,j,k\}$, the complete three-index equation is
\begin{equation}\label{eq:Ntriple}
 d_i h_{kj}=h_{ji}h_{ki},\qquad
 h_{ji}h_{kj}=d_jh_{ki},\qquad
 h_{ji}d_k=h_{ki}h_{kj}.
\end{equation}
After this normalization the triple relations are independent of $q$.
\end{proposition}

\begin{proof}
Before normalization, the three relevant equations are
$$
 (q+1)d_i k_{kj}=\rho_\eps k_{ji}k_{ki},\quad
 \rho_\eps k_{ji}k_{kj}=(q+1)d_jk_{ki},\quad
 (q+1)k_{ji}d_k=\rho_\eps k_{ki}k_{kj},
$$
where $\rho_0=q$ and $\rho_1=1$.  Since $\gm_\eps=(q+1)/\rho_\eps$, these are exactly the displayed relations.
Every remaining residual coefficient is a scalar multiple of one of these
three relations, so the system is complete.
\end{proof}

\begin{proposition}\label{prop:rank3upper}
Let $K$ be a solution and suppose $k_{ab}\ne0$ with $a<b$.  Then
$p_a=p_b=\eps$,
$$
 k_{ba}=-q^{-{\sigma_\eps}}k_{ab},\qquad k_{aa}=k_{bb}=0.
$$
If $c$ is any third index, then all entries joining $c$ to $a$ or $b$ vanish.  A diagonal entry $k_{cc}\ne0$ is compatible with the pair if and only if $p_c\ne\eps$ and $c\notin {I_{ab}}$.
\end{proposition}

\begin{proof}
The two-index calculation gives the first assertion.  Let $c$ be any third
index.  For each of the three positions of $c$ relative to
$a<b$, substitution of
$$
 k_{ba}=-q^{-{\sigma_\eps}}k_{ab},\qquad k_{aa}=k_{bb}=0
$$
into the three-index reflection equation, without the triangularity assumption,
 on $\{a,b,c\}$ gives,
after multiplication by a nonzero power of $q$, the four residual equations
$$
 (q-1)k_{ab}k_{ac}=0,\quad (q-1)k_{ab}k_{ca}=0,\quad
 (q-1)k_{ab}k_{bc}=0,\quad (q-1)k_{ab}k_{cb}=0.
$$
Thus all four entries joining $c$ to $a$ or $b$ vanish, independently of
$p_c$.  After setting them to zero, the normalized scalar multiplying
$k_{ab}k_{cc}$ in the residual coefficient involving the skew pair
and the diagonal entry is
$$
\begin{array}{c|cc}
&p_c=\eps&p_c  \ne \eps\\ \hline
c<a\ \text{or}\ b<c&q^4-1&0\\
a<c<b&q^2-1&q^2-1
\end{array}
$$
Since  $q^4\ne1$, the nonzero entries of this table cannot vanish.  Hence $k_{cc}\ne0$ is possible only when  $p_c\ne\eps$  and
$c\notin {I_{ab}}$.  Conversely, when $p_c\ne\eps$ and $c\notin I_{ab}$, the joining entries are zero and direct substitution shows that the full three-index residual of the skew pair together with $k_{cc}$ is zero.  Thus these conditions are also sufficient.
\end{proof}

\section{ Reflection equation in dimension  four}\label{sec:rank4}
By Corollary~\ref{cor:fourlocal}, the interaction between two disjoint off-diagonal supports is
determined by the reflection equation on their four endpoints.  We shall use the following three four-index statements.

Let $K$ be lower triangular, {and let} $a<b$ and $c<d$ {have four distinct
endpoints}, with $p_a=p_b$ and $p_c=p_d$.  Suppose that its principal block
on these four endpoints has precisely the two nonzero entries
$$
 K_{\{a,b,c,d\}}=k_{ba}e_{ba}+k_{dc}e_{dc},
 \qquad k_{ba}k_{dc}\ne0.
$$

\begin{proposition}\label{prop:rank4lower}
The matrix $K_{\{a,b,c,d\}}$ solves the induced twisted reflection equation if and
only if the intervals $[a,b]$ and $[c,d]$ cross.  Up to interchanging the pairs, the allowed order is
$$
 a<c<b<d.
$$
Separated and nested placements are incompatible.
\end{proposition}

\begin{proof}
Relabel the ordered endpoints as $1<2<3<4$.  In a separated or nested
placement, the coefficient $\Fc_{34}^{12}$ has two different forms.
For separated pairs $(1,2)$ and $(3,4)$, with $p_1=p_2$ and $p_3=p_4$,
it is $-s_3\omega k_{21}k_{43}$.  For nested pairs $(1,4)$ and $(2,3)$,
with $p_1=p_4$ and $p_2=p_3$, it is
$(-1)^{p_1p_2}\omega k_{41}k_{32}$.
Neither coefficient can vanish because $\omega\ne0$.  In the crossing placement the entire
two-entry residual is zero.
\end{proof}

Let $K$ be a solution, and let $a<b$ and $c<d$ be four distinct
indices.  Suppose $k_{ab}\ne0$ and $p_a=p_b=\eps$.
Lemma~\ref{lem:twoall} gives
$$
 k_{ba}=-q^{-{\sigma_\eps}}k_{ab},\qquad k_{aa}=k_{bb}=0.
$$
\begin{proposition}\label{prop:rank4upper}
{In the four-dimensional reflection equation, the following compatibility conditions hold for} the entries on $c,d$:
 \begin{enumerate}[label=\textup{(\roman*)}]
\item a skew pair on $c,d$ with $p_c=p_d=\eps$ is compatible exactly
when the intervals $[a,b]$ and $[c,d]$ are separated;
\item a lower-only  entry $k_{dc}\ne0$ with $k_{cd}=0$ and $p_c=p_d$  is
incompatible in all relative positions;
\item a skew pair on $c,d$ with $p_c=p_d\ne\eps$ is incompatible
with the skew pair on $(a,b)$.
\end{enumerate}
\end{proposition}

\begin{proof}
Proposition~\ref{prop:rank3upper} removes all entries joining the first pair to
the two new endpoints, so it remains to consider only the interaction
of the two disjoint supports.  The table records a normalized scalar multiplying
$k_{ab}k_{dc}$ in the lower-only row and $k_{ab}k_{cd}$ in the
skew-pair rows.  A zero means that every residual coefficient for the corresponding configuration vanishes.
$$
\begin{array}{c|ccc}
\text{second support}&\text{separated}&\text{crossing}&\text{nested}\\ \hline
\text{lower-only, equal  endpoint parity }
 &q^2-1&q^2-1&q^2-1\\
\text{skew pair, endpoint parity }\eps
 &0&q^2-1&q^2-1\\
\text{skew pair, endpoint parity } \ne \eps
 &q^4-1&q^2-1&q^2-1
\end{array}
$$
The assumption $q^4\ne1$ makes every displayed factor nonzero.  The table proves
(i)--(iii).
\end{proof}

\begin{proposition}\label{prop:22obsnew}
Let $K$ be a lower-triangular solution, and let
$i_1<i_2<i_3<i_4$ contain two indices of parity $0$ and two of parity $1$.
Then
$$
 \begin{cases}
 k_{i_3i_2}k_{i_4i_1}=0,&p_{i_1}=p_{i_2},\\
 k_{i_2i_1}k_{i_4i_3}=0,&p_{i_1}\ne p_{i_2}.
 \end{cases}
$$
In each line both factors connect indices of opposite parity and hence are
odd entries.  In particular,
$\S_-$ cannot contain all six lower pairs $(i_r,i_s)$, $s<r$, on
these four indices.
\end{proposition}

\begin{proof}
The diagonal entries play no role in this obstruction.  Their relations with
the lower support are recorded separately in Lemma~\ref{lem:allclosure}; the
two coefficients used below contain only the displayed off-diagonal products.
Here the parity word means the four-tuple $(p_{i_1},p_{i_2},p_{i_3},p_{i_4})$.
If $p_{i_1}=p_{i_2}$, the coefficient $\Fc_{i_4i_3}^{i_1i_2}$ is
$$
 2k_{i_3i_2}k_{i_4i_1}\quad{\text{for }0011},
 \qquad
 -2k_{i_3i_2}k_{i_4i_1}\quad{\text{for }1100}.
$$
If $p_{i_1}\ne p_{i_2}$, the parity word is one of
$0101,0110,1001,1010$.  The coefficient $\Fc_{i_2i_4}^{i_1i_3}$ is
$$
 2k_{i_2i_1}k_{i_4i_3}
 \quad\text{for }0101,0110,
 \qquad
 -2k_{i_2i_1}k_{i_4i_3}
 \quad\text{for }1001,1010.
$$
The residual vanishes, so the asserted products are zero.  The two endpoints
of each displayed factor have opposite parities.
\end{proof}

\section{Classification of K-matrices}\label{sec:structure}
{In this section we} assemble the local restrictions into the full classification {of
K-matrices}.  The three subsections {address}, respectively, even lower-triangular solutions, solutions with a nonzero upper entry, and non-even lower-triangular solutions.  The three branches are collected in
Theorem~\ref{thm:structureall} at the end of the section.

\begin{lemma}\label{lem:upperforces}
Let $K$ be a solution of \eqref{eq:TRE}.  If
$\S_+\ne\varnothing$, then $K$ is even.
\end{lemma}

\begin{proof}
Choose $(a,b)\in\S_+$.  Lemma~\ref{lem:twoall} shows that
$p_a=p_b=\eps$, and the $p_a=p_b$ branch of the two-index equations gives
$$
 k_{ba}=-q^{-\sigma_\eps}k_{ab},\qquad k_{aa}=k_{bb}=0.
$$
Assume that an odd entry $k_{ji}\ne0$ also occurs.  By
Lemma~\ref{lem:twoall} it lies below the diagonal and both of its endpoint
diagonals are nonzero.  One endpoint, say $c$, has parity $\eps$.  If
$c=a$ or $c=b$, this already contradicts $k_{aa}=k_{bb}=0$
by Lemma~\ref{lem:twoall}(ii).  Otherwise,
restrict  the RE to the three indices $a,b,c$.  Proposition~\ref{prop:rank3upper}
forces $k_{cc}=0$, again a contradiction.  Hence a solution with an upper
entry has no odd entries.
\end{proof}

Thus every non-even solution satisfies $\S=\Lc$, so it is lower
triangular.  It remains to determine its possible mixed lower support.

\subsection{Even lower-triangular solutions}
\begin{lemma}\label{lem:separator}
Let $K$ be an even solution of \eqref{eq:TRE} and suppose $k_{rr}\ne0$.
Then $k_{ii}=0$ for every $i$ with $p_i\ne p_r$.  If, in addition,
$i<j$ satisfy $p_i=p_j\ne p_r$, then
\begin{enumerate}[label=\textup{(\roman*)}]
\item $i<r<j$ {implies} $k_{ij}=k_{ji}=0$;
\item $r<i<j$ or $i<j<r$ {implies}
$k_{ji}=-q^{-s_i}k_{ij}$.
\end{enumerate}
Consequently, if $K$ is lower triangular and $k_{rr}\ne0$ with
$p_r=\eps$, then $k_{ij}=0$ whenever $p_i=p_j\ne\eps$.
\end{lemma}

\begin{proof}
The first statement follows from Lemma~\ref{lem:two}, applied to the pair
$\{i,r\}$.  The other assertions involve only the indices $i,r,j$.  When
$i<r<j$, two coefficients of the three-index residual, normalized by powers of $q$, are
$$
 (q^2-1)k_{rr}k_{ij},\qquad (q^2-1)k_{rr}k_{ji}.
$$
When $r<i<j$ or $i<j<r$,  there is a normalized three-index residual coefficient
$$
 (q^2-1)k_{rr}(k_{ji}+q^{-s_i}k_{ij}).
$$
Since $k_{rr}\ne0$, these coefficients give (i) and (ii).  In the lower-triangular case $k_{ij}=0$, and either relation then forces $k_{ji}=0$ as well.
\end{proof}

Let $K$ be even and lower triangular. If $K$ solves \eqref{eq:TRE}
and $k_{rr}\ne0$, put $\eps=p_r$. Lemma~\ref{lem:separator} then gives
$k_{ij}=0$ whenever $p_i=p_j\ne\eps$. Since $K$ is even, its
mixed-parity entries vanish. Hence its support is in
$I_\eps\times I_\eps$.

\begin{lemma}\label{lem:reduction}
Fix $\eps\in\{0,1\}$. Suppose that the support of $K$ is in
$I_\eps\times I_\eps$. In the notation of Lemma~\ref{lem:locality},
let $K_{I_\eps}$ denote the principal block of $K$ on $I_\eps$. Then
$$
 \Fc_V(K)=0\quad\Longleftrightarrow\quad
 \Fc{_{V_{I_\eps}}}(K_{I_\eps})=0.
$$
\end{lemma}

\begin{proof}
The forward implication is Lemma~\ref{lem:locality}, applied to the
principal subspace $V_{I_\eps}$.  The reverse implication is
Lemma~\ref{lem:zeroext}, since $K$ is the zero extension of $K_{I_\eps}$ to $V$.
\end{proof}

For this principal block, we use the normalization
\begin{equation}\label{eq:hdef}
 h_{ij}=\gm_\eps^{-1}k_{ij}\quad(j<i),\qquad d_i=k_{ii}.
\end{equation}
By Proposition~\ref{prop:rank3hom}, the triple relations
\eqref{eq:Ntriple} hold.

Let $i,j,k,\ell$ be four indices of parity $\eps$. For a lower-triangular
block on these four indices satisfying the normalized triple relations
\eqref{eq:Ntriple} on each of its four triples, the reflection
equation in dimension four is equivalent to
\begin{equation}\label{eq:N4}
 h_{ji}h_{\ell k}=h_{kj}h_{\ell i},
 \qquad i<j<k<\ell .
\end{equation}
\noindent Indeed, relabel the four indices as $1<2<3<4$ and put
$$
 \Delta=h_{21}h_{43}-h_{32}h_{41}.
$$
Denote by $\mathcal A_4$ the polynomial algebra
 $$
 \mathcal A_4=\C[d_1,d_2,d_3,d_4,
 h_{21},h_{31},h_{32},h_{41},h_{42},h_{43}],
$$
and let $\mathcal I_3\subset\mathcal A_4$ be the ideal generated by the normalized relations
\eqref{eq:Ntriple} on each of the four
triples.  Modulo $\mathcal I_3$, every residual coefficient is zero except
for the following coefficient and its negative:
$$
 \Fc_{34}^{12}(K)\equiv
 \begin{cases}
 q^{-3}(1-q)(1+q)^3\Delta,&\eps=0,\\
 -q^{-1}(1-q)(1+q)^3\Delta,&\eps=1.
 \end{cases}
 \pmod{\mathcal I_3}.
$$
Thus the four-index residual forces \eqref{eq:N4}; conversely, imposing
$\Delta=0$ makes every residual coefficient vanish modulo $\mathcal I_3$.
The {multipliers $q^{-3}(1-q)(1+q)^3$ and
$-q^{-1}(1-q)(1+q)^3$ are nonzero under the standing assumption on $q$}.

The four-index locality  principle then asserts that the relations \eqref{eq:Ntriple}, together with \eqref{eq:N4},
are equivalent to the reflection equation for lower-triangular matrices on any set of
indices of parity $\eps$, with the inherited order.

\begin{lemma}\label{lem:closure}
Assume that \eqref{eq:Ntriple} and \eqref{eq:N4} hold
on a set of indices of parity $\eps$, with the inherited order.  For
$i<j<k$, the following closure rules hold.
\begin{enumerate}[label=\textup{(\roman*)}]
\item A nonzero diagonal at one vertex together with a nonzero edge joining the other two vertices forces all three edges and all three diagonals on the triple to be nonzero (a marked clique).
\item A nonzero diagonal at one vertex together with two nonzero edges incident with that vertex likewise forces a marked clique on the triple.
\end{enumerate}
\end{lemma}

\begin{proof}
For (i), if $d_i h_{kj}\ne0$, then the first identity in
\eqref{eq:Ntriple} gives
$h_{ji}h_{ki}\ne0$; substituting these nonzero factors into
the second and third identities in \eqref{eq:Ntriple} gives $d_jd_k\ne0$.  The cases
$d_jh_{ki}\ne0$ and $d_kh_{ji}\ne0$ are identical, using
the second and third identities in \eqref{eq:Ntriple}, respectively.

For (ii), if for instance $d_i$, $h_{ji}$ and $h_{ki}$ are nonzero,
then the first identity in \eqref{eq:Ntriple} forces $h_{kj}\ne0$, after which (i) applies.
The cases with the marked vertex $j$ or $k$ are obtained from
the second or third identity in \eqref{eq:Ntriple}.
\end{proof}

For $S\subseteq I_\eps$, $|S|\ge3$, and nonzero parameters $x_i$ define a clique matrix
\begin{equation}\label{eq:dense}
 L_{\eps,S}(x)=
 \sum_{i\in S}x_i^2e_{ii}
 +\gm_\eps\sum_{\substack{i,j\in S\\j<i}}x_ix_j e_{ij}.
\end{equation}
 For $a<b$ with $p_a=p_b$ define
\begin{equation}\label{eq:two-block}
 B_{ab}(A,B,C)=Ae_{aa}+Be_{ba}+Ce_{bb}.
\end{equation}

\begin{proposition}\label{prop:homdiag}
Assume $p_1=\cdots=p_d=\eps$.  Let $K$
be a lower-triangular solution of \eqref{eq:TRE} with at least one nonzero diagonal entry.  Then it belongs
to exactly one of the following three families:
\begin{enumerate}[label=\textup{(D\arabic*)}]
\item $K$ is diagonal;
\item $K=B_{ab}(A,B,C)$ for some $a<b$, with $B\ne0$ and $(A,C)\ne(0,0)$;
\item $K=L_{\eps,S}(x)$ for a set $S$ with $|S|\ge3$ and all $x_i\ne0$.
\end{enumerate}
Conversely every matrix in these families solves \eqref{eq:TRE}.
\end{proposition}

\begin{proof}
The displayed families satisfy the normalized triple equations 
\eqref{eq:Ntriple} and
\eqref{eq:N4}, hence solve \eqref{eq:TRE} by the four-index locality principle.  Conversely, use the marked vertices and edges introduced before
Lemma~\ref{lem:allclosure}.  With no edge the solution is diagonal, while
an edge with exactly two active vertices gives \eqref{eq:two-block}.  If at
least three active vertices occur, choose a marked one.  An edge avoiding
that vertex yields a marked clique on three vertices by
Lemma~\ref{lem:closure}(i).  If every edge is incident with {the chosen marked
vertex}, either two such edges exist and Lemma~\ref{lem:closure}(ii) applies,
or a third active vertex is marked and lies off the unique edge, so part (i)
applies again.  Once a marked clique on three vertices has been
obtained, parts (i) and (ii) attach {every other} marked vertex to the whole
clique.  An
unmarked active vertex lies on an edge; choosing a vertex of this clique distinct
from the two endpoints and applying part (i) first marks both endpoints.
Induction on the number of active vertices shows that the active set $S$ is
a marked clique.

Choose $i_0=\min S$, $x_{i_0}^2=d_{i_0}$, and set
$x_j=h_{ji_0}/x_{i_0}$.  For each $j\ne i_0$, choose a third vertex of
$S$;  combining the first identity in \eqref{eq:Ntriple}
with the second or third according to the position of the third vertex,
and cancelling the
nonzero edge entries gives
$h_{ji_0}^2=d_{i_0}d_j$.  Hence $x_j^2=d_j$, and for $i_0<j<k$,
$$
 h_{kj}=\frac{h_{ji_0}h_{ki_0}}{d_{i_0}}=x_kx_j.
$$
Thus $K=L_{\eps,S}(x)$.
\end{proof}

A collection of pairs
\begin{equation}\label{eq:crossmatch}
 M=\{(a_1,b_1),\ldots,(a_t,b_t)\}
\end{equation}
is called a \emph{matching} if {all its endpoints are distinct.  It is called
a \emph{crossing matching} if}, after numbering the pairs appropriately,
\begin{equation}\label{eq:crossorder}
 a_1<a_2<\cdots<a_t<b_1<b_2<\cdots<b_t.
\end{equation}
Thus every two intervals $[a_r,b_r]$ and $[a_s,b_s]$ cross.  We call $M$ parity-compatible if
\begin{equation}\label{eq:parcompat}
 p_{a_r}=p_{b_r}\qquad(1\le r\le t).
\end{equation}
Define
\begin{equation}\label{eq:crossK}
 C_M(c)=\sum_{r=1}^t c_r e_{b_r a_r},\qquad c_r\ne0\quad(1\le r\le t),
\end{equation}
with $C_\varnothing=0$.

\begin{proposition}\label{prop:strictlower}
An even strictly lower-triangular matrix solves \eqref{eq:TRE} if and only if it is $C_M(c)$ for a parity-compatible crossing matching $M$.
\end{proposition}

\begin{proof}
By the four-index locality principle, it is enough to check the reduced reflection equation on the
endpoints involved.  On three indices $i<j<k$, two nonzero lower entries
sharing a vertex imply $p_i=p_j=p_k$.  With zero diagonal, the three-index
relations \eqref{eq:cl} reduce, up to nonzero scalar
factors, to
$$
k_{ji}k_{ki}=0,\qquad k_{ji}k_{kj}=0,\qquad k_{ki}k_{kj}=0.
$$
Hence no two support edges can share an endpoint, so the support is a
matching in the sense defined above.

For two support edges, the preceding three-index argument shows that their
endpoints are disjoint.  The principal block on these four endpoints therefore
contains exactly those two support entries.  Proposition~\ref{prop:rank4lower}
then excludes separated and nested intervals.  Hence every two support intervals cross.

Conversely, a principal submatrix of $C_M(c)$ on at most four indices contains at most two support edges, and if two occur then they cross.  Each of these local configurations solves the equation, so Corollary~\ref{cor:fourlocal} gives the result in arbitrary rank.
\end{proof}

\begin{theorem}\label{thm:lower}
An even lower-triangular matrix $K$ solves \eqref{eq:TRE} if and only if one of the following holds.
\begin{enumerate}[label=\textup{(L\arabic*)}]
\item All diagonal entries vanish and $K=C_M(c)$ is a parity-compatible crossing matching as in \eqref{eq:crossK}.
\item Some diagonal entry is nonzero.  Then there is a parity $\eps$ such
that every nonzero entry of $K$ has both indices in $I_\eps$ and, with
the induced order on $I_\eps$, $K$ is {either} a diagonal matrix, {or} a
two-point block \eqref{eq:two-block} with $B\ne0$ and
$(A,C)\ne(0,0)$, or a clique matrix \eqref{eq:dense} with all $x_i\ne0$.
\end{enumerate}
\end{theorem}

\begin{proof}
The zero-diagonal case is Proposition~\ref{prop:strictlower}.  Otherwise choose $r$ with $k_{rr}\ne0$ and put $\eps=p_r$.  Lemma~\ref{lem:separator} eliminates every entry $k_{ij}$ with $p_i=p_j\ne\eps$, while Lemma~\ref{lem:reduction} reduces the reflection equation to the one for a lower-triangular matrix on $V_{I_\eps}$, whose retained basis vectors all have parity $\eps$.  Proposition~\ref{prop:homdiag} then gives the three families listed there.
\end{proof}

\subsection{Solutions with upper entries}
This is the branch in which
$\S_+\ne\varnothing$.
A collection
$$
 M=\{(a_1,b_1),\ldots,(a_t,b_t)\}
$$
is called a \emph{separated matching} if
\begin{equation}\label{eq:seporder}
 a_1<b_1<a_2<b_2<\cdots<a_t<b_t.
\end{equation}

\begin{lemma}\label{lem:isolation}
Let $K$ be a solution and suppose $k_{ab}\ne0$ with $a<b$.  Then $K$ is
even, $p_a=p_b=\eps$, and
\begin{equation}\label{eq:pairratio}
 k_{ba}=-q^{-\sigma_\eps}k_{ab},\qquad k_{aa}=k_{bb}=0.
\end{equation}
Moreover:
\begin{enumerate}[label=\textup{(\roman*)}]
\item every entry joining $a$ or $b$ to any third index is zero, and
a nonzero diagonal entry $k_{rr}$ can occur only when
$p_r\ne\eps$ and $r\notin I_{ab}$;
\item every lower-only off-diagonal entry on two other indices is zero;
\item every {other nonzero} off-diagonal entry belongs to a pair satisfying
\eqref{eq:pairratio}, {whose} endpoints {are} of parity $\eps$, and the pair
intervals are separated.
\end{enumerate}
\end{lemma}

\begin{proof}
Lemma~\ref{lem:upperforces} guarantees evenness, and Lemma~\ref{lem:two} gives
\eqref{eq:pairratio}.  Proposition~\ref{prop:rank3upper} implies (i).  On any
two {other} same-parity endpoints, Lemma~\ref{lem:twoall}(i)
says that the block is either lower-only or, if an upper entry is present, a
skew pair.  Proposition~\ref{prop:rank4upper}(ii) excludes every lower-only
case.  Parts (i) and (iii) of that proposition respectively force {another}
$\eps$-pair to be separated and exclude a pair of the opposite parity.  Mixed
endpoints are already excluded by evenness.  This proves (ii) and (iii).
\end{proof}

Fix $\eps\in\{0,1\}$.  Let $M$ be a nonempty separated matching whose endpoints all have parity $\eps$.  Let
\begin{equation}\label{eq:Dallowed}
 D\subseteq I_{1-\eps}\setminus\bigcup_{(a,b)\in M}I_{ab}.
\end{equation}
For nonzero parameters $z_{ab}$ and $y_r$ define
\begin{equation}\label{eq:upperfamily}
 P_{\eps;M,D}=
 \sum_{(a,b)\in M}z_{ab}
 \bigl(e_{ab}-q^{-\sigma_\eps}e_{ba}\bigr)
 +\sum_{r\in D}y_r e_{rr}.
\end{equation}

\begin{theorem}\label{thm:upper}
An even solution of \eqref{eq:TRE} satisfies
$\S_+\ne\varnothing$ if and only if it is a matrix
$P_{\eps;M,D}$ of the form \eqref{eq:upperfamily}.
\end{theorem}

\begin{proof}
Necessity follows from Lemma~\ref{lem:isolation}.  Conversely, write
$P_{\eps;M,D}=\sum_\alpha X_\alpha$, where each $X_\alpha$ is either one
skew pair or one diagonal entry.  Since $\Fc$ is quadratic,
$$
 \Fc\!\left(\sum_\alpha X_\alpha\right)
 =\sum_\alpha \Fc(X_\alpha)
 +\sum_{\alpha<\beta}\bigl(\Fc(X_\alpha+X_\beta)
 -\Fc(X_\alpha)-\Fc(X_\beta)\bigr).
$$
Thus it is enough to check each piece and each sum of two pieces.  A single
piece solves the equation by the two-index calculation and zero extension;
two diagonal pieces solve it for the same reason, since their indices have the same parity $1-\eps$.  A skew pair together
with a diagonal piece is compatible by Proposition~\ref{prop:rank3upper},
using $p_r\ne\eps$ and $r\notin I_{ab}$ from \eqref{eq:Dallowed}.  Finally,
two skew pairs are compatible by Proposition~\ref{prop:rank4upper}(i),
since their intervals are separated and their endpoints have parity
$\eps$.  Hence every term in the displayed quadratic decomposition
vanishes, and so $\Fc(P_{\eps;M,D})=0$.
\end{proof}

\subsection{Non-even lower-triangular solutions}
\begin{lemma}\label{lem:activeclique}
Suppose a lower-triangular solution has a nonzero odd entry.
Then both endpoints of that entry have nonzero diagonal entries.  If the
active set contains more than two indices, it is a marked clique.
\end{lemma}

\begin{proof}
Lemma~\ref{lem:twoall} marks both endpoints of an odd support edge, so
they form an initial marked clique $A$.  Suppose inductively that
$A$ is a marked clique and let $c$ be another active vertex.  If
$c$ is marked, apply Lemma~\ref{lem:allclosure} to $c$ and an edge of
$A$; this joins $c$ to both endpoints of that edge.  Applying the same rule
with the new edge from $c$ and each remaining marked vertex of $A$ joins
$c$ to all of $A$.

{If $c$ is an unmarked active vertex,} it lies on an edge $c-\!\!\!-d$.  Choose a marked vertex of
$A$ different from $d$ (possible because $A$ contains the two endpoints of
the original odd support edge).  Lemma~\ref{lem:allclosure} applied
to this marked vertex and the edge $c-\!\!\!-d$ marks $c$ and joins it to that
vertex.  The preceding paragraph then joins $c$ to all of $A$.  Repeating
the argument exhausts the active set.
\end{proof}

For $a<b$ with $p_a\ne p_b$ and $x_a,x_b\in\C^\times$ define {a} mixed
two-point block
\begin{equation}\label{eq:Qab}
 Q_{ab}(x_a,x_b)=
 x_a^2e_{aa}+\eta x_bx_a e_{ba}+x_b^2e_{bb}.
\end{equation}
The choice of sign in the middle entry is absorbed by changing the sign of either $x_a$ or $x_b$.

Now let $S$ be {a} set of at least three indices.  Suppose that all
indices in $S$ except one, denoted by $r$, have a common parity $\eps\ne p_r$.
We refer to $r$ as the \emph{minority index} and to $\eps$ as
the \emph{majority parity}.  For nonzero $x_i$, $i\in S$, define a one-minority clique matrix
\begin{align}
 M_{\eps,S,r}(x)
 ={}&\sum_{i\in S}x_i^2e_{ii}
 +\gm_\eps
   \sum_{\substack{j<i\\i,j\in S\setminus\{r\}}}
   x_ix_j e_{ij}\notag\\
 &+\eta\sum_{\substack{j\in S\\j<r}}x_rx_j e_{rj}
 +\sigma_\eps\eta\sum_{\substack{i\in S\\r<i}}x_ix_r e_{ir}.
 \label{eq:minoritydense}
\end{align}

\begin{lemma}\label{lem:minoritydense}
Every matrix \eqref{eq:minoritydense} solves \eqref{eq:TRE}.  Conversely,
let a lower-triangular solution have an active set $S$ with $|S|\ge3$
that is a marked clique with both parities present.  Then $S$ has a unique minority index and,
after a choice of the square root $\eta$ in \eqref{eq:eta}, the solution is
of the form \eqref{eq:minoritydense}.
\end{lemma}
\begin{proof}

By Proposition~\ref{prop:22obsnew}, a marked clique cannot contain two
disjoint pairs of indices of different parities. Since both parities are present, one of them
occurs exactly once at a vertex of $S$; call it $r$.

{Assume first that $K$ is such a lower-triangular solution.  Choose} $x_r$ with
$x_r^2=k_{rr}$.  If $i<r$ is a majority
index, define
$$
 x_i=\frac{k_{ri}}{\eta x_r},
$$
whereas for $r<i$ put
$$
 x_i=\frac{k_{ir}}{\sigma_\eps\eta x_r}.
$$
The mixed two-index relation {\eqref{eq:mixed2local}} and
$\eta^2=-\omega$ give $x_i^2=k_{ii}$.  Now take $r$ together with two
majority indices $i<j$.  In each of the three possible positions of $r$, {the
three-index residual contains a normalized coefficient of the form}
$$
 (q^2-1)x_r^2\bigl(k_{ji}-\gm_\eps x_jx_i\bigr)=0.
$$
Thus
$$
 k_{ji}=\gm_\eps x_jx_i
 \qquad(i<j,\quad i,j\ne r).
$$
This is exactly \eqref{eq:minoritydense}.

{It remains to verify the first statement of the lemma.  Consider} a principal
block of {a matrix \eqref{eq:minoritydense} of} size at most four.  Either all
of its indices have the majority parity $\eps$, or it contains the single
minority index.  Substitution in the two- and three-index residuals gives
zero using $\eta^2=-\omega$.  In the four-index case, with three
{majority} indices and one minority index, a direct
expansion of the residual leaves nine potentially nonzero coefficients before
the relation $\eta^2=-\omega$ is imposed.  Each has the form
$$
 c(q)x_r^2x_ix_j(\eta^2+\omega),
 \qquad
 c(q)\in\{\pm(1+q),\ \pm(1+q^{-1}),\ \pm(q+q^{-1})\},
$$
where the majority indices $i,j$ may coincide.  Hence all nine coefficients
vanish.   The four-index locality principle gives the full equation.
\end{proof}

The parametrizations $L_{\eps,S}(x)$, $Q_{ab}(x_a,x_b)$, and
$M_{\eps,S,r}(x)$ are unchanged by simultaneously replacing all their
$x$-parameters with their negatives.  For nonzero parameters this is the only
sign redundancy once $\eta$ is fixed.

\begin{theorem}\label{thm:noneven}
{A} solution $K$ of \eqref{eq:TRE} is non-even if and only if it
belongs to one of the following two lower-triangular families.
\begin{enumerate}[label=\textup{(N\arabic*)}]
\item A mixed two-point block $Q_{ab}(x_a,x_b)$ in
\eqref{eq:Qab}, with $p_a\ne p_b$.
\item A one-minority clique matrix $M_{\eps,S,r}(x)$ in
\eqref{eq:minoritydense}, with $|S|\ge3$.
\end{enumerate}
All displayed $x$-parameters are nonzero.
\end{theorem}

\begin{proof}
Let $K$ be non-even.  By Lemma~\ref{lem:upperforces} it has no upper entry, hence is lower triangular.  Choose a nonzero odd lower entry.  If its endpoints are the only active indices, the mixed two-index relation \eqref{eq:mixed2local} gives $Q_{ab}$.  Otherwise Lemma~\ref{lem:activeclique} shows that the active set is a marked clique; Proposition~\ref{prop:22obsnew} leaves a unique minority index, and Lemma~\ref{lem:minoritydense} gives the one-minority clique matrix.

Conversely, $Q_{ab}$ satisfies the mixed two-index relation because $\eta^2=-\omega$, and Lemma~\ref{lem:zeroext} embeds it into arbitrary rank.  The second family is covered by Lemma~\ref{lem:minoritydense}.
\end{proof}

We can summarize the results as follows.

\begin{theorem}\label{thm:structureall}
Let $K\in\End(V)$ solve \eqref{eq:TRE}.  Exactly one of the following holds.
\begin{enumerate}[label=\textup{(\Roman*)}]
\item
$\S_+\ne\varnothing$.  Then $K$ is automatically even and is a matrix $P_{\eps;M,D}$ from \eqref{eq:upperfamily}.  Its pairs form a nonempty separated matching whose endpoints lie in $I_\eps$, and each pair has the form
$$
 z(e_{ab}-q^{-\sigma_\eps}e_{ba}).
$$
The remaining entries are arbitrary $k_{rr}$ with
$p_r\ne\eps$ and $r\notin I_{ab}$ for every \mbox{$(a,b)\in M$}.
\item
$\S=\Lc$.  Then $K$ is lower triangular.  If $K$ is even, it is exactly one of the lower families in Theorem~\ref{thm:lower}: a parity-compatible crossing matching with zero diagonal, or a diagonal matrix, a two-point block, or a  clique matrix, with the latter three types supported on indices of one parity.  If $K$ is non-even, it is exactly one of the two families in Theorem~\ref{thm:noneven}: a mixed two-point block or a one-minority clique matrix.
\end{enumerate}
Every listed matrix is, conversely, a solution of \eqref{eq:TRE}.
\end{theorem}

\section{Invertible solutions}\label{sec:invnew}

Let us compute the ranks of the obtained K-matrices.

\begin{lemma}\label{lem:familyranks}
With the parameter restrictions imposed above,
$$
 \operatorname{rank}C_M=|M|,\qquad
 \operatorname{rank}L_{\eps,S}=|S|,\qquad
 \operatorname{rank}P_{\eps;M,D}=2|M|+|D|,
$$
and
$$
 \operatorname{rank}B_{ab}(A,B,C)=
 \begin{cases}
 2,&AC\ne0,\\
 1,&AC=0,
 \end{cases}
 \qquad
 \operatorname{rank}Q_{ab}=2,
 \qquad
 \operatorname{rank}M_{\eps,S,r}=|S|.
$$
\end{lemma}

\begin{proof}
The matrices $L_{\eps,S}$, $Q_{ab}$, and $M_{\eps,S,r}$ are lower
triangular with nonzero diagonal on their active indices.  The matrix
$P_{\eps;M,D}$ is a direct sum of nonsingular skew-pair blocks and
nonzero scalar blocks, while the nonzero entries of $C_M$ occupy distinct
rows and columns.  Finally, $\det B_{ab}=AC$; if $AC=0$, the assumptions
$B\ne0$ and $(A,C)\ne(0,0)$ give rank one.
\end{proof}

For $\eps\in\{0,1\}$ let $\Ru_\eps$ be the {set of maximal} intervals
$[a,b]$ contained in $I_\eps$.  We call these intervals \emph{runs} and take the
length of $[a,b]$ to be $b-a+1$.  Write
$$
 E_\eps:\quad \text{every run in }\Ru_\eps\text{ has even length}.
$$

\begin{theorem}\label{thm:globalinv}
Let $K$ be an invertible solution of \eqref{eq:TRE}.
\begin{enumerate}[label=\textup{(\Alph*)}]
\item If $p_1=\cdots=p_d=\eps$, then:
\begin{itemize}
\item for $d=1$, $K$ is an arbitrary nonzero scalar;
\item for $d=2$, either $K=\left(\begin{smallmatrix}A&0\\B&C\end{smallmatrix}\right)$ with $AC\ne0$, or $K=z\left(\begin{smallmatrix}0&1\\-q^{-\sigma_\eps}&0\end{smallmatrix}\right)$;
\item for $d\ge3$, $K$ is an arbitrary invertible diagonal matrix, the clique matrix $L_{\eps,\{1,\ldots,d\}}(x)$ supported on all indices, or, when $d$ is even, the perfect skew matching
$$
 \sum_{r=1}^{d/2}z_r(e_{2r-1,2r}-q^{-\sigma_\eps}e_{2r,2r-1}).
$$
\end{itemize}
In part (A), $B$ is arbitrary in the lower $2\times2$ family, whereas $z$,
all $z_r$, and all $x_i$ in $L_{\eps,\{1,\ldots,d\}}(x)$ are nonzero.

\item Suppose the grading is mixed.  The invertible even family pairing the
indices in $I_\eps$ exists exactly when $E_\eps$ holds, and is
\begin{equation}\label{eq:invmixed}
 K^{(\eps)}=
 \sum_{p_i\ne\eps}y_i e_{ii}
 +\sum_{[a,b]\in\Ru_\eps}
 \sum_{r=0}^{(b-a-1)/2}z_{a+2r}
 \bigl(e_{a+2r,a+2r+1}-q^{-\sigma_\eps}e_{a+2r+1,a+2r}\bigr),
\end{equation}
with all $y$- and $z$-parameters distinct from zero.

\item The non-even invertible solutions occur only in mixed grading: for
$d=2$ there is $Q_{12}$, while for $d\ge3$ they exist exactly when one of
$I_0,I_1$ has cardinality one.  If $r$ is its unique element and $\eps$ is the common parity of all remaining indices,
then $K=M_{\eps,\{1,\ldots,d\},r}$.
\end{enumerate}
There are no other invertible solutions.
\end{theorem}

\begin{proof}
Apply Theorem~\ref{thm:structureall} and Lemma~\ref{lem:familyranks}.  When all basis vectors have parity $\eps$, rank $d$ leaves the full diagonal and clique families,
together with the arbitrary lower $2\times2$ block; an upper matching is
invertible only when it is the adjacent perfect matching.  For a mixed even
solution the lower branch is singular, so an invertible matrix must come from
the family $P_{\eps;M,D}$.  Every pair is then contained in a
run contained in $I_\eps$ and all indices in $I_\eps$ must be paired.  Thus each
such run has even length and is paired adjacently, which gives
\eqref{eq:invmixed}.  The ranks of $Q_{ab}$ and $M_{\eps,S,r}$ give the
non-even cases.
\end{proof}

The Satake interpretation of the  even families is discussed in
Subsection~\ref{sec:satakelower}.

\subsection{Satake data for even invertible solutions}\label{sec:satakelower}
We first indicate a relation between invertible even K-matrices and graded Satake diagrams.
We will not turn to the theory of  coideal subalgebras in quantum groups and confine ourselves to a  classical consideration. In the classical limit $q\to 1$,
the matrix $K$ tends to an even matrix $K_0$ that  is fixed by a subalgebra $\mathfrak{k}$ in the general linear Lie superalgebra $\mathfrak{g}$,
with respect to the ``twisted'' action:
\[
 \rho(x) K_0+K_0\rho(x)^t=0\qquad(x\in\mathfrak{k}),
\]
where $\rho$ is the natural representation of $\mathfrak{g}$ on $V$ restricted to $\mathfrak{k}$.
Such $\mathfrak{k}$ is associated with a Satake diagram in the standard way \cite{RegelskisVlaar}.

In the case of constant parity sequence $I=I_\epsilon$, K-matrices are known
to correspond to two diagrams \cite{NS}; see also \cite[Section~3.1.2]{AMSsatake}:
\begin{equation}\label{Satake-diagrams}
 \begin{picture}(45,10)
\put(2,3){\circle{3}}
\put(3,3){\line(1,0){5}}
\put(11,0){$\cdots$}
\put(26,3){\line(1,0){5}}
\put(32,3){\circle{3}}
 \end{picture}
\quad\quad
 \begin{picture}(45,10)
\put(2,3){\circle*{3}}
\put(3,3){\line(1,0){5}}
\put(11,0){$\cdots$}
\put(26,3){\line(1,0){5}}
\put(32,3){\circle{3}}
\put(34,3){\line(1,0){7}}
\put(42,3){\circle*{3}}
  \end{picture}
\end{equation}
where the graph on the right comprises an odd number of nodes of alternating colour with black nodes at the ends.
Diagonal $K$ are associated with the diagram on the left while the diagram on the right gives rise to
 even dimensional $K$.
In the classical limit $q\to 1$, such a $K$-matrix tends to a block-diagonal matrix
\begin{equation}
K_0=\sum_{i=1}^{d/2} z_{2i-1}(e_{2i-1,2i}-e_{2i,2i-1}),
\label{classical2by2}
\end{equation}
while a diagonal $K$ tends to diagonal.

In the case of mixed grading, there is a family of Satake diagrams that consist of
alternating even blocks of  types (\ref{Satake-diagrams}) separated by odd roots \cite[Section~2.2]{AMSsatake}.
The block of white roots can be empty. Empty even blocks can appear both at the ends and in the middle of the diagram.
The black-and-white subdiagrams correspond to subspaces of even dimension where the matrix $K_0$ restricts to
(\ref{classical2by2}). On the complementary subspace $K_0$  is diagonal.

Next we relate the Satake data
with
the lower-triangular even families.
Assume that $p_1=\cdots=p_d=\eps$, so every simple root
$\alpha_i=\zeta_i-\zeta_{i+1}$ is even.  Consider the all-white  diagram on the left
in
(\ref{Satake-diagrams}).
Since the black subdiagram is empty and $\tau(\alpha_i)=\alpha_i$, the additional
Cartan mixture term allowed in \cite{AMSsatake} may occur at every
node.  With $u_{\alpha_i}=h_{\alpha_i}$, the classical generators are
\begin{equation}\label{eq:satakegen}
 x_{\alpha_i}=e_{\alpha_i}+c_i f_{\alpha_i}
 +\grave{c}_i h_{\alpha_i},\qquad 1\le i<d.
\end{equation}
Here  $c_i$ and $\grave{c}_i$ are  two complex mixture parameters.
Let $\mathfrak{k}\subset \mathfrak{g}=\mathfrak{sl}_d$ be the Lie subalgebra
generated by the elements in \eqref{eq:satakegen}.

Put
\begin{equation}\label{eq:Lclassical}
 L_0(x)=\lim_{q\to1}L_{\eps,\{1,\ldots,d\}}(x)
 =\sum_{i=1}^d x_i^2e_{ii}
 +2\sum_{1\le j<i\le d}x_ix_j e_{ij}.
\end{equation}
For $r_i=x_{i+1}/x_i$, set
\begin{equation}\label{eq:satakemixL}
 \grave{c}_i=-r_i,\qquad
 c_i=-\sigma_\eps r_i^2,
 \qquad 1\le i<d.
\end{equation}
Thus $c_i=-\sigma_\eps\grave{c}_i^{\,2}$.  In the normalization used in
\cite[Section~4.2]{AMSsatake}, the natural representation is
\begin{equation}\label{eq:satakenormalization}
 \rho(e_{\alpha_i})=\sigma_\eps e_{i,i+1},\qquad
 \rho(f_{\alpha_i})=e_{i+1,i},\qquad
 \rho(h_{\alpha_i})=\sigma_\eps(e_{ii}-e_{i+1,i+1}).
\end{equation}
This convention comes from the invariant form: for two consecutive weights of
parity $\eps$, one has $\alpha_i(h_{\alpha_i})=2\sigma_\eps$.
Consequently
\begin{equation}\label{eq:satakematrixL}
 \rho(x_{\alpha_i})=\sigma_\eps
 \bigl(-r_i e_{ii}+e_{i,i+1}-r_i^2e_{i+1,i}
       +r_i e_{i+1,i+1}\bigr).
\end{equation}
 The nonzero common factor  $\sigma_\eps$ in
\eqref{eq:satakematrixL} can be canceled.
Indeed, if $u_i=e_i+r_i e_{i+1}$ and $v_i=-r_i e_i+e_{i+1}$, then $\rho(x_{\alpha_i})=\sigma_\eps u_i v_i^t$, while $L_0v_i=-x_ix_{i+1}u_i$ and $v_i^tL_0=x_ix_{i+1}u_i^t$; hence $\rho(x_{\alpha_i})L_0+L_0\rho(x_{\alpha_i})^t=0$.  Thus $L_0(x)$ is a twisted invariant of $\mathfrak{k}$ for the
mixture parameters \eqref{eq:satakemixL}.
The subalgebra $\mathfrak{k}$ is isomorphic to the electrical Lie algebra of type~A \cite{LamPylyavskyy}, realized as $\mathfrak{sl}_d^{(\mathbf a)}$ in \cite{BGG}, with
$a_i=-r_i=-x_{i+1}/x_i$.

The other lower-triangular invertible families have the same all-white Satake
origin at the classical level.  For an invertible diagonal matrix
$D=\sum_i d_i e_{ii}$, take $\grave{c}_i=0$ and
\begin{equation}\label{eq:satakemixdiag}
 c_i=-\sigma_\eps\frac{d_{i+1}}{d_i}.
\end{equation}
For $d=2$, an arbitrary invertible lower-triangular matrix
$\left(\begin{smallmatrix}A&0\\ B&C\end{smallmatrix}\right)$ has the one-node
all-white data
\begin{equation}\label{eq:satakemix2}
 \grave{c}_1=-\frac{B}{2A},\qquad
 c_1=-\sigma_\eps\frac{C}{A}.
\end{equation}
For $A=x_1^2$, $B=2x_1x_2$, and $C=x_2^2$, these values agree with
\eqref{eq:satakemixL}.  When $d=1$ there are no simple roots in $\mathfrak{g}$, so the twisted
invariance condition is vacuous and the nonzero scalar family corresponds to
the empty diagram.

Thus every invertible even family in Theorem~\ref{thm:globalinv} is
accounted for by a type A graded Satake diagram.
The paired families are associated with the shaft diagrams (in the classification of \cite[Section~3.1.2]{AMSsatake}); for constant parity, the classical limits of
the lower-triangular families are twisted invariants of the all-white shaft
diagram \cite{AMSsatake}.

We revisit the electrical Lie algebra appearing above as the even
part of the classical stabilizer of the full-support non-even family.
We turn to this next.

\section{Classical stabilizer of the non-even family}\label{sec:superelectrical}

\subsection{Classical limit and the alternating form}
We now turn to the full-support non-even family.  Let
$K=M_{\eps,\{1,\ldots,d\},r}(x)$, so $r$ is the unique minority index, and put
$n=d-1$.  Write the majority indices as $i_1<\cdots<i_n$, identify
$U=\operatorname{span}\{v_{i_1},\ldots,v_{i_n}\}$ with $\C^n$ in this order,
and set
\[
 y_a=x_{i_a},\qquad
 \mathbf{y}=(y_1,\ldots,y_n)^t,\qquad c=x_r^2,\qquad
 L=\sum_{a=1}^n y_a^2e_{aa}
   +2\sum_{1\le b<a\le n}y_ay_be_{ab}.
\]
Thus $L$ is the classical limit of the homogeneous even clique  matrix on the
majority subspace $U$.  Since $\eta^2=-(q-q^{-1})$, one has $\eta\to0$ as
$q\to1$.  Hence, after grouping the two parity subspaces, the classical limit
$M_0=\lim_{q\to1}K$ is
\begin{equation}\label{eq:Mclassicalblocks}
 M_0=
 \begin{cases}
  \left(\begin{smallmatrix}L&0\\0&c\end{smallmatrix}\right),
    &\eps=0,\quad V=U_{\bar0}\oplus\C_{\bar1}v_r,\\[4pt]
  \left(\begin{smallmatrix}c&0\\0&L\end{smallmatrix}\right),
    &\eps=1,\quad V=\C_{\bar0}v_r\oplus U_{\bar1}.
 \end{cases}
\end{equation}

Put
\[
 W=\ker(\mathbf{y}^t)\subset\C^n,
 \qquad
 \mathbf a=(a_1,\ldots,a_{n-1}),
 \qquad
 a_i=-\frac{y_{i+1}}{y_i}.
\]
Introduce a bilinear form on $W$ by
\[
 \Omega(z,w)=z^tLw,
 \qquad
 \widehat\Omega=-\sigma_\eps c^{-1}\Omega.
\]

\begin{lemma}\label{lem:Momega}
The form $\Omega$ is alternating on $W$.  If $n$ is odd, it is
nondegenerate.  If $n$ is even, it has a one-dimensional kernel generated by
\[
 r_0=\bigl((-1)^{i-1}y_i^{-1}\bigr)_{i=1}^n.
\]
\end{lemma}

\begin{proof}
Since $L+L^t=2\mathbf{y}\mathbf{y}^t$, the restriction of $\Omega$ to $W$
is alternating.  Put $D=\operatorname{diag}(y_1,\ldots,y_n)$.  With the change of variables $\widetilde z=Dz$ and $\widetilde w=Dw$, the form becomes the restriction to
$\mathbf{1}^{\perp}$ of the skew matrix
$J=(\operatorname{sgn}(i-j))_{i,j=1}^n$.  An elementary induction on $m=[n/2]$ gives
rank $n-1$ on this subspace when $n$ is odd and rank $n-2$ when $n$ is
even.  In the latter case the displayed vector generates the kernel in the
original coordinates.
\end{proof}

\subsection{The classical stabilizer}
Denote
\[
 \mathfrak{k}(M_0)=\{X\in\mathfrak{gl}(V):XM_0+M_0X^t=0\}.
\]
Since $M_0$ is even, $\mathfrak{k}(M_0)$ is a Lie
superalgebra.

\begin{lemma}\label{lem:Mstabeven}
The even part of $\mathfrak{k}(M_0)$ acts trivially on $v_r$, and
\begin{equation}\label{eq:Mstabeven}
 \mathfrak{k}(M_0)_{\bar0}\big|_U
 =\{LC:C^t=C,\ C\mathbf{y}=0\}.
\end{equation}
Furthermore, the even part is the electrical Lie algebra of type~A \cite{LamPylyavskyy}, realized as $\mathfrak{sl}_n^{(\mathbf a)}$ in \cite{BGG}, of dimension $n(n-1)/2$.
\end{lemma}

\begin{proof}
An even element acts by a matrix $A$ on $U$ and by a scalar on $\C v_r$.
The invariance equation gives
\[
 AL+LA^t=0
\]
and makes the scalar vanish.  Write $D=\operatorname{diag}(y_1,\ldots,y_n)$
and $L=DTD$, where $T$ has diagonal entries $1$, entries $2$ below the
diagonal, and zero entries above it.  Thus
$T+T^t=2\mathbf{1}\mathbf{1}^t$.  On writing
$A=DBD^{-1}$, the equation becomes $BT+TB^t=0$.  Adding its transpose shows
$B\mathbf{1}=0$.  If $B=TC'$, then $C'\mathbf{1}=0$, and the same equation,
together with $T+T^t=2\mathbf{1}\mathbf{1}^t$, gives
\[
 (C'-C'^t)T=-2C'^t\mathbf{1}\mathbf{1}^t.
\]
The left-hand side has even rank unless it vanishes, whereas the right-hand side has
rank at most one.  Hence $C'=C'^t$.  Transforming back gives precisely
\eqref{eq:Mstabeven}.  Conversely, if $C^t=C$ and $C\mathbf{y}=0$, then
\[
 (LC)L+L(LC)^t=LC(L+L^t)=2LC\mathbf{y}\mathbf{y}^t=0.
\]
Thus the dimension in \eqref{eq:Mstabeven} is $n(n-1)/2$.

After the above identification of $U$, the matrices of the form
\eqref{eq:satakematrixL} are nonzero scalar multiples of
\[
 e_i+a_i h_i-a_i^2f_i.
\]
They therefore generate the type~A electrical Lie algebra \cite{LamPylyavskyy}, in the realization $\mathfrak{sl}_n^{(\mathbf a)}$ of \cite[Definition~1.1]{BGG}.  That algebra has dimension $n(n-1)/2$ by
\cite[Theorem~1.3]{BGG}; since it fixes $L$, it equals
\eqref{eq:Mstabeven}.  Taking the trace of $AL+LA^t=0$ after multiplication
by $L^{-1}$ gives $\operatorname{tr}A=0$.
\end{proof}

For $z\in W$, set
\begin{equation}\label{eq:Mstabodd}
 \Theta_z=
 \begin{cases}
  \left(\begin{smallmatrix}0&-c^{-1}Lz\\z^t&0\end{smallmatrix}\right),
    &\eps=0,\\[5pt]
  \left(\begin{smallmatrix}0&z^t\\c^{-1}Lz&0\end{smallmatrix}\right),
    &\eps=1.
 \end{cases}
\end{equation}

\begin{lemma}\label{lem:Mstabodd}
The map $z\mapsto\Theta_z$ identifies $W$ with
$\mathfrak{k}(M_0)_{\bar1}$.  For $z,w\in W$,
\begin{equation}\label{eq:Mstabbracket}
 [\Theta_z,\Theta_w]\big|_U
 =-\sigma_\eps c^{-1}L(zw^t+wz^t).
\end{equation}
Consequently,
\[
 \dim\mathfrak{k}(M_0)_{\bar1}=n-1,
 \qquad
 [\mathfrak{k}(M_0)_{\bar1},\mathfrak{k}(M_0)_{\bar1}]
 =\mathfrak{k}(M_0)_{\bar0}.
\]
\end{lemma}

\begin{proof}
Direct block matrix calculation in the two decompositions of
\eqref{eq:Mclassicalblocks} gives \eqref{eq:Mstabodd}.  For $z,w\in W$,
their supercommutator acts trivially on $v_r$ and gives
\eqref{eq:Mstabbracket}.  The matrices $zw^t+wz^t$ span all symmetric
matrices $C$ satisfying $C\mathbf{y}=0$, so \eqref{eq:Mstabbracket} and
Lemma~\ref{lem:Mstabeven} give the last assertion.
\end{proof}

\subsection{A Lie superalgebra model}
In this subsection, we clarify the Lie superalgebra structure of the stabilizer $\mathfrak{k}(M_0)$.
Introduce a graded vector space 
\[
 \mathfrak{S}(W,\widehat\Omega)=W\odot W\oplus W
\]
 with even part $W\odot W$ and  odd part $W$.
Define a $\Z_2$-graded bilinear operation $[-,-]$ on $ \mathfrak{S}(W,\widehat\Omega)$   setting
\[
 [u,v]=u\mathbin{\odot}v,
 \qquad
 [u\mathbin{\odot}v,z]
 =\widehat\Omega(v,z)u+\widehat\Omega(u,z)v,
 \qquad u,v,z\in W,
\]
where $u\mathbin{\odot}v=u\otimes v+v\otimes u$ and extending it to an even--even bracket 
by the induced action on $W\odot W$.

Define linear maps
\[
 \Phi_{\bar0}(z\mathbin{\odot}w)
 =-\sigma_\eps c^{-1}L(zw^t+wz^t),
 \qquad
 \Phi_{\bar1}(z)=\Theta_z
\]
with values in $\End(V)$.
They amount to a graded linear map 
$$\Phi=\Phi_{\bar0}\oplus\Phi_{\bar1}\colon \mathfrak{S}(W,\widehat\Omega)\to \End(V)$$
 with zero kernel.
\begin{proposition}\label{prop:Mstabmodel}
The map $\Phi$ preserves the brackets and takes $\mathfrak{S}(W,\widehat\Omega)$ onto $\mathfrak{k}(M_0)$ thereby establishing an isomorphism
\[
 \mathfrak{S}(W,\widehat\Omega)\simeq\mathfrak{k}(M_0)
\]
of  Lie
superalgebras.
\end{proposition}
{
\begin{proof}
The spanning statement in the proof of Lemma~\ref{lem:Mstabodd} shows that
$\Phi_{\bar0}$ is a vector-space isomorphism from $W\odot W$ onto
$\mathfrak{k}(M_0)_{\bar0}$, while \eqref{eq:Mstabodd} gives the same
conclusion for $\Phi_{\bar1}$.  Equation \eqref{eq:Mstabbracket} gives the
odd--odd bracket.  Moreover, for $u,v,z\in W$, the stabilizer action on the
odd parameters is
\[
 [\Phi_{\bar0}(u\mathbin{\odot}v),\Theta_z]
 =\Theta_{\widehat\Omega(v,z)u+\widehat\Omega(u,z)v}.
\]
The even--even bracket is the induced action on $W\odot W$.  Hence
$\Phi$ is a linear isomorphism preserving all brackets.
Transporting the supercommutator through $\Phi$ proves the graded Jacobi
identity on $\mathfrak{S}(W,\widehat\Omega)$ and the claimed isomorphism.
\end{proof}

\begin{corollary}\label{cor:Mstabilizer}
The stabilizer $\mathfrak{k}(M_0)$ is a Lie subsuperalgebra of
$\mathfrak{sl}(V)$.  If $n$ is odd, then
\[
 \mathfrak{k}(M_0)\simeq\mathfrak{osp}(1|n-1).
\]
If $n$ is even, then the skew-symmetric form $\widehat\Omega$ is degenerate with a one-dimensional kernel.  The even part of $\mathfrak{k}(M_0)$ is
$\mathfrak{sp}_{n-1}$ in the sense of \cite[Theorem~1.7]{BGG}, and $\mathfrak{k}(M_0)$ has a one-dimensional even center.  In both cases it is a natural
Lie-superalgebra extension of the electrical Lie algebra in
\eqref{eq:Mstabeven}.
\end{corollary}

\begin{proof}
The even part has  zero trace  by Lemma~\ref{lem:Mstabeven}, while the odd part
is off-diagonal in \eqref{eq:Mstabodd}; hence
$\mathfrak{k}(M_0)\subset\mathfrak{sl}(V)$.  If $n$ is odd,
Lemma~\ref{lem:Momega} makes $\widehat\Omega$ nondegenerate, and
Proposition~\ref{prop:Mstabmodel} gives the orthosymplectic Lie superalgebra
$\mathfrak{osp}(1|n-1)$.

If $n$ is even, Lemma~\ref{lem:Momega} gives the kernel $\C r_0$.  The
defining brackets of $\mathfrak{S}(W,\widehat\Omega)$ show that
$r_0\mathbin{\odot}r_0=[r_0,r_0]$ spans a central even ideal.  Under
Proposition~\ref{prop:Mstabmodel}, the even part is the algebra
$\mathfrak{sp}_{n-1}$ described in \cite[Theorem~1.7]{BGG}, namely the
subalgebra of $\mathfrak{sp}_n$ annihilating a vector; the one-dimensional
center is noted in the paragraph following that theorem.
\end{proof}
}

\section{The  symmetric grading $(1^m,0^N,1^m)$}\label{sec:symmetricnew}
Assume $m,N>0$.  There are two odd runs, both of length $m$, and one even run of length $N$.
Theorem~\ref{thm:globalinv} specializes as follows.

\begin{corollary}\label{cor:symmetricpaper}
For the grading $(1^m,0^N,1^m)$ with $m,N>0$:
\begin{enumerate}[label=\textup{(\roman*)}]
\item the even family pairing the odd indices, corresponding to \cite[(4.31)]{AMS}, exists if and only if $m$ is even;
\item the even family pairing the even indices, corresponding to \cite[(4.32)]{AMS}, exists if and only if $N$ is even;
\item in addition, a non-even family of one-minority clique matrices supported on all indices exists if and only if $N=1$;
\item if both $m$ and $N$ are odd with $N>1$, there is no invertible solution.
\end{enumerate}
When both $m$ and $N$ are even, the two even families coexist; when $N=1$ and $m$ is even, the odd-paired even family coexists with the non-even family of one-minority clique matrices.
\end{corollary}

\begin{proof}
The first two assertions are the run-length conditions in Theorem~\ref{thm:globalinv}.  Since $|I_1|=2m$, one of $I_0,I_1$ has cardinality one only when $N=1$; this is exactly the non-even case in part (iii).  If $m$ and $N$ are both odd with $N>1$, neither even family exists and the non-even family of one-minority clique matrices supported on all indices is also absent.
\end{proof}

If $m=0$ or $N=0$, all basis vectors have the same parity and
Theorem~\ref{thm:globalinv}(A) applies.

\section{Singular solutions}\label{sec:singnew}
\begin{theorem}\label{thm:globalsing}
The singular solutions are exactly:
\begin{enumerate}[label=\textup{(S\arabic*)}]
\item singular diagonal matrices supported on $I_\eps\times I_\eps$
for some $\eps\in\{0,1\}$;
\item two-point lower blocks $B_{ab}(A,B,C)$ with $B\ne0$,
$(A,C)\ne(0,0)$, and either $d>2$ or $AC=0$;
\item clique matrices $L_{\eps,S}(x)$ with $S\subseteq I_\eps$ and $|S|<d$;
\item all nonempty parity-compatible crossing matchings $C_M(c)$;
\item matrices $P_{\eps;M,D}$ with $2|M|+|D|<d$;
\item mixed two-point blocks $Q_{ab}$ when $d>2$;
\item one-minority clique matrices $M_{\eps,S,r}$ with $|S|<d$.
\end{enumerate}
\end{theorem}

\begin{proof}
Apply Theorem~\ref{thm:structureall} and Lemma~\ref{lem:familyranks}.
Crossing matchings are always singular; every other family is singular
precisely when its rank is smaller than $d$.
\end{proof}

\noindent
\underline{\large \bf Acknowledgement}

\vspace{10pt}
\noindent
This work is done at the Center of Pure Mathematics MIPT. It is financially supported by Russian Science Foundation grant 26-11-00115.

\noindent
D. Algethami is  thankful to the Deanship of Graduate Studies and Scientific Research at University of Bisha for supporting this work through the Fast-Track Research Support Program.

\vspace{10pt}
\noindent
\underline{\large \bf Data Availability}

\vspace{10pt}

 Data sharing not applicable to this article as no datasets were generated or analysed during the current study.

\subsection*{Declarations}

\underline{\large \bf Competing interests}

\vspace{10pt}
\noindent
The authors have no competing interests to declare that are relevant to the content of this article.


\begin{thebibliography}{A}
\setlength{\itemsep}{2pt}
\bibitem{Cherednik} I. V. Cherednik: {\em Factorizing particles on a half-line and root systems}, Theoret. Math. Phys., {\bf 61} (1984), 977--983.

\bibitem{Sklyanin} E. K. Sklyanin: {\em Boundary conditions for integrable quantum systems}, J. Phys. A, {\bf 21} (1988), 2375--2389.

\bibitem{KulishSklyanin} P. P. Kulish, E. K. Sklyanin: {\em Algebraic structure related to the reflection equation}, J. Phys. A, {\bf 25} (1992), 5963--5975.

\bibitem{KulishSasakiSchwiebert} P. P. Kulish, R. Sasaki, C. Schwiebert: {\em Constant solutions of reflection equations and quantum groups}, J. Math. Phys., {\bf 34} (1993), 286--304.

\bibitem{Letzter} G. Letzter: {\em Symmetric pairs for quantized enveloping algebras}, J. Algebra, {\bf 220} (1999), 729--767.

\bibitem{KolbRE} S. Kolb: {\em Quantum symmetric pairs and the reflection equation}, Algebr. Represent. Theory, {\bf 11} (2008), 519--544.

\bibitem{Kolb} S. Kolb: {\em Quantum symmetric Kac--Moody pairs}, Adv. Math., {\bf 267} (2014), 395--469.

\bibitem{BalagovicKolb} M. Balagovi\'c, S. Kolb: {\em Universal K-matrix for quantum symmetric pairs}, J. Reine Angew. Math., {\bf 747} (2019), 299--353.

\bibitem{RegelskisVlaar} V. Regelskis, B. Vlaar: {\em Quasitriangular coideal subalgebras of $U_q(\mathfrak g)$ in terms of generalized Satake diagrams}, Bull. Lond. Math. Soc., {\bf 52} (2020), 693--715.

\bibitem{AppelVlaar} A. Appel, B. Vlaar: {\em Universal K-matrices for quantum Kac--Moody algebras}, Represent. Theory, {\bf 26} (2022), 764--824.

\bibitem{Yamane} H. Yamane: {\em Quantized enveloping algebras associated with simple Lie superalgebras and their universal R-matrices}, Publ. Res. Inst. Math. Sci., {\bf 30} (1994), 15--87.

\bibitem{Arnaudon} D. Arnaudon, J. Avan, N. Cramp\'e, A. Doikou, L. Frappat, E. Ragoucy: {\em General boundary conditions for the $\mathfrak{sl}(N)$ and $\mathfrak{sl}(M|N)$ open spin chains}, J. Stat. Mech. Theory Exp. (2004), P08005.

\bibitem{LimaSantos} A. Lima-Santos: {\em Reflection matrices for the $U_q[\mathfrak{sl}(m|n)^{(1)}]$ vertex model}, J. Stat. Mech. Theory Exp. (2009), P08006.

\bibitem{Shen} Y. Shen: {\em Quantum supersymmetric pairs and $\imath$Schur duality of type AIII}, J. Algebra, {\bf 661} (2025), 853--903.

\bibitem{ShenWang} Y. Shen, W. Wang: {\em Quantum supersymmetric pairs of basic types}, Commun. Math. Phys., {\bf 406} (2025), Paper 187.

\bibitem{AMS} D. Algethami, A. Mudrov, V. Stukopin: {\em Quantum super-spherical pairs}, J. Algebra, {\bf 674} (2025), 276--313.

\bibitem{AMSsatake} D. Algethami, A. Mudrov, V. Stukopin: {\em Graded Satake diagrams and super-symmetric pairs}, arXiv:2604.03394 [math.QA] (2026).

\bibitem{AMS24} D. Algethami, A. Mudrov, V. Stukopin: {\em Solutions to graded reflection equation of $GL$-type}, Lett. Math. Phys., {\bf 114} (2024), 22.

\bibitem{Zhang} R. B. Zhang: {\em Universal $L$ operator and invariants of the quantum supergroup $U_q\allowbreak(\mathrm{gl}(m/n))$}, J. Math. Phys., {\bf 33} (1992), 1970--1979.

\bibitem{M1} A. Mudrov: {\em Characters of the $U_q(\mathfrak{gl}(n))$-reflection equation algebra}, Lett. Math. Phys., {\bf 60}(3) (2002), 283--291.

\bibitem{NS} M. Noumi, T. Sugitani: {\em Quantum symmetric spaces and related $q$-orthogonal polynomials}, in Group Theoretical Methods in Physics (ICGTMP), World Scientific, River Edge, NJ, 1995, 28--40.

\bibitem{NDS} M. Noumi, M. S. Dijkhuizen, T. Sugitani: {\em Multivariable Askey--Wilson polynomials and quantum complex Grassmannians}, Fields Inst. Commun., {\bf 14} (1997), 167--177.

\bibitem{DoninMudrovEQ} J. Donin, A. Mudrov: {\em Reflection equation, twist, and equivariant quantization}, Israel J. Math., {\bf 136} (2003), 11--28.

\bibitem{LamPylyavskyy} T. Lam, P. Pylyavskyy: {\em Electrical networks and Lie theory}, Algebra Number Theory, {\bf 9}(6) (2015), 1401--1418.

\bibitem{GorbounovTalalaev} V. Gorbounov, D. Talalaev: {\em Electrical varieties as vertex integrable statistical models}, J. Phys. A: Math. Theor., {\bf 53} (2020), 454001.

\bibitem{BGG} A. Berenstein, A. Gainutdinov, V. Gorbounov: {\em Generalized electrical Lie algebras}, Adv. Math., {\bf 478} (2025), 110405.


\end{thebibliography}
\end{document}